\expandafter\def\csname input@path\endcsname{%
    {VMSTA/sty/},
    {VMSTA/img/},
    {VMSTA/bib/}
}

\RequirePackage{xparse,expl3}
\documentclass[numbers,compress]{vmsta2}

\volume{0}
\issue{0}
\pubyear{2026}
\articletype{research-article}

\usepackage[utf8]{inputenc}
\usepackage[T2A,T1]{fontenc}
\usepackage[ukrainian,english]{babel}
\usepackage{dsfont}

\DeclareMathOperator{\Dom}{Dom}

\newtheorem{theorem}{Theorem}[section]
\newtheorem{lemma}{Lemma}[section]
\newtheorem{corollary}{Corollary}[section]
\newtheorem{definition}{Definition}[section]
\newtheorem{example}{Example}[section]

\newtheorem{remark}{Remark}[section]

\begin{document}

\begin{frontmatter}

\pretitle{Research Article}

\title{Integrability of a composition of functions}

\begin{aug}
\author[a]{\inits{A.}\fnms{Alexander }~\snm{Kukush}\thanksref{cor1}\ead[label=e1]{akukush@imath.kiev.ua}}
\thankstext[type=corresp,id=cor1]{Corresponding author.}
\author[b]{\inits{M.  }\fnms{Mykhailo }~\snm{Naumov}\ead[label=e2]{mnaoumov@gmail.com}}

\address[a]{\institution{Institute of Mathematics of NAS of Ukraine}, \cny{Ukraine}}
\address[b]{\institution{Volia Software}, \cny{Mexico}}
\end{aug}

\begin{abstract}
The paper deals with the Riemann and Lebesgue integrability of functions. Sufficient conditions and criteria are stated for the integrability of a composition of functions. Relevant counterexamples are constructed based on  Cantor's perfect sets and the Cantor staircase function. The main tool of research is Lebesgue's criterion for Riemann integrability.
\end{abstract}



\end{frontmatter}

\section{Introduction}

Let $f:[a,b] \to [c,d]$, $f$ be a Riemann integrable function on $[a,b]$ and $G$ be a continuous function on $[c,d]$, then the composition $G\circ f$ is Riemann integrable; see \cite{Olmsted}, p. 153, Problem 55. This is a direct consequence of Lebesgue's criterion for Riemann integrability \cite{Bogachev}, stated in Theorem \ref{criteria} below.

On the other hand, a composition of  Riemann integrable functions need not be such.  A counterexample is constructed in Gelbaum and Olmsted (2003), Chap. 4, Sect. 9, where $f$ is the Riemann   popcorn function  equal to $1/q$ if $x\in [0, 1]$ is rational and expressed as $p/q$ in the lowest terms and equal to zero if $x\in [0, 1]$ is irrational, and $G$ is the indicator  of $(0,1]$.  Both functions are Riemann integrable in $[0,1]$, but $G\circ f$ is not.

The integrability of a composition of Riemann integrable functions was studied specifically by Klippert (1992) \cite{Klippert} and Lu (1999) \cite{Lu}. Both papers prove that $G\circ f$ is Riemann integrable when $f$ is Riemann integrable and $G$ is continuous, and both show by counterexamples that continuity of the internal function is not enough: in \cite{Klippert} the internal function is continuous and strictly increasing, and in \cite{Lu} it is continuous and vanishes on a set of positive measure that contains no interval. Further, Klippert \cite{Klippert} shows that $G\circ f$ is Riemann integrable if $f$ is continuous and non-decreasing and $G$ is bounded with countably many discontinuities, and that $G\circ f$ is Riemann integrable for each Riemann integrable $G$ if $f$ is monotone and differentiable with a Riemann integrable derivative satisfying $|f'|\ge m>0$. Lu \cite{Lu} states the latter conclusion, without proof, for $f$ with a continuous non-vanishing derivative.

These results are sufficient conditions. In the present paper we give criteria. For a fixed external function $G$, Theorem 2.3 shows that $G\circ f$ is Riemann integrable for each Riemann integrable $f$ if and only if $G$ is continuous. For a fixed internal function $f$, Theorems 2.5 and 2.6 show that $G\circ f$ is Riemann integrable for each Riemann integrable $G$ if and only if $f$ has the $N^{-1}$-property outside the intervals where $f$ is constant. In the counterexamples of \cite{Klippert} and \cite{Lu}, the internal function fails this property. Theorem 3.1 contains the sufficient conditions of \cite{Klippert} and \cite{Lu} for differentiable internal functions, because it requires only absolute continuity and $f'\neq 0$ almost everywhere. The sufficient condition of \cite{Klippert} with a monotone internal function follows from inclusion \eqref{eq:Dh} in the proof of Theorem 2.6. The results of Section 4 on Lebesgue integrability have no counterpart neither in \cite{Klippert}, nor  in \cite{Lu}.

It is interesting to give mild conditions for the integrability of a composition of functions. In the present paper, we mainly deal with Riemann integrability. Here, the main tool is the above-mentioned Lebesgue's criterion. In addition, we give sufficient conditions for the Lebesgue integrability of a composition. Here, it is crucial to provide that the preimage of a Lebesgue measurable set under the internal function is Lebesgue measurable as well. We construct counterexamples that show that certain conditions of statements cannot be avoided. Most of the results can be extended to functions of several variables. 

The second author reported the results in Kyiv at the 9th conference in Stochastic Processes and Probability "The Skorokhod Readings" in June 2026.

The paper is organized as follows. Section 2 provides the conditions under which a composition of functions is Riemann integrable, and some counterexamples are constructed. Section 3 involves continuous  and absolutely continuous functions. Section 4 deals with the Lebesgue integrability on the real line, and Section 5 concludes.

We give the notation used in the paper.

\begin{tabular}{@{}l@{\quad}p{0.78\textwidth}@{}}

$\mathds{1}_A$ & indicator function of a set $A$ \\
$\Dom(f)$ & the domain of $f$; for $f: A \to B, \Dom(f) = A$ \\
$f^{-1}A$ & preimage of a set $A$ under a mapping $f$ \\
$D(f)$ & the discontinuity set $\{ x \in \Dom(f)$: $f$  is discontinuous at point   $x$\} \\
$C(A)$ & the space of continuous functions on a set $A$ \\
$C[a,b]$ & the space of continuous functions on $[a,b]$ \\
$AC[a,b]$ & the space of absolutely continuous functions on $[a,b]$ \\
$R[a,b]$ & the space of Riemann integrable functions on $[a,b]$ \\
$\lambda$ & Lebesgue measure on $\mathbb{R}$ \\
$\lambda^{*}$ &  Lebesgue outer measure on $\mathbb{R}$ \\
$B(\mathbb{R})$ & sigma-algebra of Borel sets on $\mathbb{R}$ \\
$S$ &  sigma-algebra of Lebesgue measurable sets on $\mathbb{R}$ \\
$C$ & Cantor's perfect  set with $\lambda(C) = 0$ \cite{Gelbaum} Chap. 8, Section 1\\
$\theta(x)$ & the Cantor staircase function \cite{Gelbaum} Chap. 8, Section 15\\
$K$ & "fat" Cantor's perfect set with $\lambda(K) = 1/2$ \cite{Gelbaum} Chap. 8, Section 4\\
$L[a,b]$ & the space of Lebesgue integrable functions on $[a,b]$ \\
$L(A)$ & the space of Lebesgue integrable functions on a Lebesgue measu\-rable set $A$
\end{tabular}

\section{Conditions for Riemann integrability of a composition}

In Sections 2 and 3, we deal with the functions $f: [a,b] \to [c,d]$,  $G :[c,d] \to \mathbb{R}$, and $h := G \circ f: [a,b] \to \mathbb{R},$ $h$ denotes a composition of functions. Thus, $f$ is an internal function and $G$  an  external one.
We consider the conditions under which  $h \in R[a,b]$.
Our main tool is as follows.

\begin{theorem}[Lebesgue's criterion for Riemann integrability, see \cite{Bogachev}]\hfill
\label{criteria}

Let $g$ be a bounded real function on $[a,b]$.
Then $g \in R[a,b]$  if and only if $\lambda(D(g)) = 0$.

\end{theorem}

\begin{theorem}[Sufficient condition for integrability]\hfill\label{t22}

If $f\in R[a,b]$, $G$ is bounded, and $\lambda(f^{-1} D(G)) = 0$, then $h \in R[a,b]$.
\end{theorem}
\begin{proof}
 If $f$ is continuous at the point $x$ and $G$ is continuous at the point $f(x)$, then $h$ is continuous at the  point $x$. Hence
$$D(h) \subset D(f) \cup f^{-1} D(G),~
\lambda^{*}(D(h)) \leq \lambda^{*}(D(f)) + \lambda^{*}(f^{-1} D(G)).$$
Since $f \in R[a,b]$, we have $\lambda^{*}(D(f)) =\lambda(D(f))= 0$. Next,
$\lambda^{*}(f^{-1} D(G)) = 0 $ according to the condition of the theorem. Therefore,
$0=\lambda^{*}(D(h)) = \lambda(D(h)).$ This together with the boundedness of $h$ implies $h \in R[a,b]$.
\end{proof}

\begin{example}[Conditions of Theorem \ref{t22} are not necessary]\hfill
\end{example}
Let $f(x) = 0$ if  $x \in [0, \frac{1}{2}]$,  $f(x)= x - \frac{1}{2}$ if $x \in (\frac{1}{2}, 1]$, and
$G= \mathds{1}_{\{0\}}$. Then the composition
$h = \mathds{1}_{[0,\frac{1}{2}]}.$ All three functions are Riemann integrable, but $D(G)=\{0\},$
$f^{-1} D(G) = [0, \frac{1}{2}]$, and   $\lambda(f^{-1} D(G)) > 0$.

\begin{theorem}[Criterion for a fixed external function]\hfill

Fix  $G \in R[c,d]$.  For each Riemann integrable  $f:[a,b]\to [c,d]$,   $G \circ f \in R[a,b]$  if and only if  $G \in C[c,d]$.
\end{theorem}
\begin{proof}
\textit{Sufficiency} was mentioned in the Introduction.

\textit{Necessity.}
Prove by the contrary.
Assume that $G$  is discontinuous. We have to show that there exists $f \in R[a,b]$ with $h \notin R[a,b]$.

Without loss of generality  we may and do assume that $[a,b]=[c,d]=[0,1]$.
\medskip

\textit{1. Construction of $f$.}

$G \notin C[0,1]$ implies that $\exists s \in D(G)$ $  \exists \{ s_n: n \geq 1\} \subset [0,1]$ with $s_n \to s$ and $G(s_n) \to z \neq G(s)$ as $ n \to \infty $.

Let $K$ be the "fat" Cantor's set,
$[0,1] \setminus K = \cup_{k=1}^\infty (a_k, b_k) $,
$c_k := (a_k + b_k)/2$,

$ j(k) := l$ if $2^{-(l+1)} \leq b_k - a_k < 2^{-l} $,

$f(x) := \begin{cases}s, & x \in K \\ s + \frac{s_{j(k)}-s}{c_k-a_k}(x-a_k), & x \in (a_k, c_k] \\ s + \frac{s_{j(k)}-s}{c_k-b_k}(x-b_k), & x \in (c_k, b_k). \end{cases}$

Then $f \in C([0,1] \setminus K)$.

\medskip
\textit{2. Proof of the continuity of $f$ in $K$.}

Fix $x_0 \in K$ and $\varepsilon > 0$.
Then $\exists N \geq 1~ \forall n \geq N: |s_n - s| < \varepsilon.$\\
$\lim_{x \to a_k+0} f(x) = f(a_k)$, hence   $$\exists \delta'_k > 0~ \forall x \in [a_k, a_k+\delta'_k]: |f(x) - f(a_k)| < \varepsilon.$$
$\lim_{x \to b_k-0} f(x) = f(b_k)$, hence $$\exists \delta''_k > 0~ \forall x \in [b_k - \delta''_k, b_k]: |f(x) - f(b_k)| < \varepsilon.$$

Define
$\delta_0 = \begin{cases}\delta'_k, & x_0 = a_k \\ \delta''_k, & x_0 = b_k \\ 1, & \textrm{otherwise,} \end{cases}$

$A_N = \{ k \geq 1: j(k) < N \}. $Then
$\forall k \in A_N: b_k - a_k \geq 2^{-N} $.

$\sum_{k \in A_N} (b_k-a_k) \leq 1$, hence  $\#A_N \leq 2^N < \infty $.
$$\delta := \min(2^{-N}, \delta_0, \min \cup_{k \in A_N} (\{ |x_0 - a_k|, |x_0-b_k| \} \setminus \{ 0 \})), ~\delta > 0. $$

Consider arbitrary $x \in [0,1]$ with $ |x-x_0| < \delta$.

If $x \in K$ then $|f(x) - f(x_0)|=|s-s|=0  < \varepsilon$,
if $x \in (a_k, b_k)$ and $j(k) \geq N$, then  $|f(x)-f(x_0)| \leq |s_{j(k)} - s| < \varepsilon $,
and if $x \in (a_k, b_k)$ and $ j(k) < N$, then $x_0 = a_k$ or $x_0 = b_k$, and since $|x-x_0| < \delta_0,$  we get  $|f(x) - f(x_0)| < \varepsilon$.

Therefore, $f \in C[0,1]$ and $  f \in R[0,1]$.

\medskip
\textit{3. Proof that the composition is not Riemann integrable.}

For $E:=K \setminus \cup_{k \geq 1} \{ a_k, b_k \},$ it holds 
$\lambda(E) = \lambda({K}) = \frac{1}{2} $.
Fix $x_0 \in E$ and $N \geq 1$,

$B_N := \cup_{k \in A_{N+1}} [a_k, b_k]$.
$K$ is nowhere dense, therefore, $$ (x_0 - 2^{-N-1}, x_0 + 2^{-N-1}) \setminus B_N \setminus K \neq \varnothing.$$
\begin{multline*}
(x_0 - 2^{-N-1}, x_0 + 2^{-N-1}) \setminus B_N \setminus K \\
= \cup_{k \geq 1} (((x_0 - 2^{-N-1}, x_0 + 2^{-N-1}) \setminus B_N) \cap (a_k, b_k)).
\end{multline*}
$$\exists k = k(N): ((x_0 - 2^{-N-1}, x_0 + 2^{-N-1}) \setminus B_N) \cap (a_k, b_k) \neq \varnothing,$$ hence
$ j(k) \geq N + 1$, and   $|b_k - a_k| \leq 2^{-N-1}$.

Select $d_k \in ((x_0 - 2^{-N-1}, x_0 + 2^{-N-1}) \setminus B_N) \cap (a_k, b_k)$. We have
$$|a_k - x_0| \leq |a_k - d_k| + |d_k - x_0| < b_k - a_k + 2^{-N-1} \leq 2 \cdot 2^{-N-1} = 2^{-N}.$$
Similarly $|b_k - x_0| < 2^{-N}$, hence
$(a_k, b_k) \subset (x_0 - 2^{-N}, x_0 + 2^{-N})$, and\\
$c_{k(N)} \in (x_0 - 2^{-N}, x_0 + 2^{-N})$.
Now,
\begin{align*}
|h(c_{k(N)}) - h(x_0)| &= |G(s_{j(k(N))}) - G(s)| \geq |z - G(s)| - |z - G(s_{j(k(N))})| \\
&\to |z - G(s)| > 0 \quad \text{as } N \to \infty.
\end{align*}
So, we get $ x_0 \in D(h)$. Thus, $  E \subset D(h)$, $  \lambda(D(h)) \geq \lambda(E) = 1/2 > 0$, and  $h$ is not Riemann integrable.
\end{proof}

\begin{definition}
Let $A\subset\mathbb{R}, f:A\to\mathbb{R}$. The function $f$ has $N^{-1}$-property if for each $N$ with $\lambda(N)=0$, it holds $\lambda(f^{-1}N)=0.$
\end{definition}

Properties of such functions are studied by Kowalczyk and Turowska (2016) \cite{Kowalczyk}. If a function $g$ is Lebesgue integrable in $\mathbb{R}$ and has $N^{-1}$-property, then it is $(S,S)$-measurable, i.e., for each $A\in S,~ g^{-1}A\in S.$

\begin{theorem}[Sufficient condition for fixed internal function]\label{t24}\hfill

Let $G$ be Riemann integrable. If $f$ has $N^{-1}$-property,
then $h \in R[a,b]$.
\end{theorem}
\begin{proof}
$G \in R[c,d]$, hence $ \lambda(D(G)) = 0$.
Since $f$ has $N^{-1}$-property, $\lambda({f^{-1} D(G)}) = 0$, and by \textrm{Theorem} \ref{t22}  $h \in R[a,b]$.
\end{proof}

\begin{remark}
The conditions of Theorem \ref{t24} are not necessary. E.g., the function $f(x)=const$ has no $N^{-1}$-property, but for each Riemann integrable $G$, it holds $G\circ f\in R[a,b].$
\end{remark}

\begin{theorem}[Criterion for a fixed internal function]\label{t25}
\hfill

Fix $f \in R[a,b]$.
If  additionally there is no interval where $f$ takes a constant value, then
for each $G \in R[c,d], G \circ f \in R[a,b]$ if and only if $f$ has $N^{-1}$-property.

\end{theorem}

\begin{proof}
\textit{Sufficiency} was proven in Theorem \ref{t24}.

\textit{Necessity}. Prove by contradiction: assuming that $f$ does not possess  $N^{-1}$-property, we construct $G \in R[c,d]$ with $h = G \circ f \notin R[a,b]$.

Thus, suppose that there exists a set $N \subset [c,d]$ with $\lambda(N) = 0$ and $\lambda^{*}(f^{-1} N) > 0$. We first replace $N$ by a Borel set with the same two properties. For each $k \geq 1$, by the outer regularity of the Lebesgue measure (\cite{Bogachev}, Chapter 1.4), choose an open set $U_{k} \supset N$ with $\lambda(U_{k}) < 1/k$, and put $H= \bigcap_{k \geq 1} U_{k}$. Then $H$ is  $G_\delta$ set (hence Borel), $N \subset H$, and $\lambda(H) \leq \inf_{k} \lambda(U_{k}) = 0$. Since $N \subset H$ implies $f^{-1} N \subset f^{-1} H$, we have $\lambda^{*}(f^{-1} H) \geq \lambda^{*}(f^{-1} N) > 0$. Changing $N$ by $H$, we may and do assume that  $N$ is Borel with $\lambda(N) = 0$.

\medskip
\textit{The support.} Let $\mu := \lambda f^{-1}$ be the induced measure on $[c,d]$, $\mu(E) = \lambda(f^{-1} E)$ for a Borel set $E$; it is a finite Borel measure of the total mass $b-a$. Since $\mu(N) = \lambda(f^{-1} N) > 0$ while $\lambda(N) = 0$, $\mu$ is not absolutely continuous with respect to $\lambda$. By the Lebesgue decomposition, $\mu = \mu_{ac} + \mu_{s}$ with $\mu_{s} \perp \lambda$ and $\mu_{s} \neq 0$; let $Z$ be a Borel set with $\lambda(Z) = 0$ on which $\mu_{s}$ is concentrated. By the inner regularity of $\mu$ (\cite{Bogachev}, Chapter 1.4), there is a compact set $F \subset Z$ with $\mu(F) > 0$. Put $B = f^{-1} F$. Then $F$ is closed, $\lambda(F) = 0$, and $\lambda(B) = \mu(F) =: \gamma > 0$.

If $M \subset F$ is any Borel set, then $G := \mathds{1}_{M} \in R[c,d]$, because $D(\mathds{1}_{M}) = \partial M \subset \overline{M} \subset F$ (since $F$ is closed) and $\lambda(F) = 0$. Moreover $h = \mathds{1}_{M} \circ f = \mathds{1}_{f^{-1} M}$, so $D(h) = \partial (f^{-1} M)$. Hence, by Theorem \ref{criteria}, it suffices to find a Borel $M \subset F$ with $\lambda(\partial (f^{-1} M)) > 0$.

 Now, we produce $M$, distinguishing some cases according to $B$.

\textit{Case 1: $\lambda(B \setminus B^{\circ}) > 0$.} Take $M = F$, so $f^{-1} M = B$. $\lambda(\partial B) \geq \lambda(B \setminus B^{\circ}) > 0$, hence $h = \mathds{1}_{B} \notin R[a,b]$.

\textit{Case 2: $\lambda(B \setminus B^{\circ}) = 0$.} Write the open set $B^{\circ} = \cup_{i} J_{i}$ as a disjoint union of open intervals and fix one component $J := J_{i}$; then $f(J) \subset F$. Consider a finite measure $\mu_{J} := (\lambda|_{J}) f^{-1}$ on $F$, of the total mass $\lambda(J) > 0$.

\textit{Case 2(a): $\mu_{J}$ has an atom $y_{0}$.} We have $\lambda(f^{-1}(y_{0}) \cap J) = \mu_{J}(\{ y_{0} \}) > 0$. A level set is as follows $P := f^{-1} \{ y_{0} \}$. If $P^{\circ} \neq \varnothing$ then $\exists (\alpha, \beta) \subset P^{\circ} \subset [a,b]: f|_{(\alpha, \beta)} = y_0$, contrary to the condition of the Theorem. We have $\lambda(\partial P) = \lambda(\overline{P}) \geq \lambda(P) > 0$. Take $M = \{ y_{0} \} \subset F$; then $G = \mathds{1}_{\{ y_{0} \}} \in R[c,d]$ and $h = \mathds{1}_{P} \notin R[a,b]$.

\textit{Case 2(b): $\mu_{J}$ is non-atomic.} Now, every level set is null in $J$, that is $\lambda(J \cap f^{-1}(y)) = \mu_{J}(\{ y \}) = 0$ for each $y$. Hence $f$ takes uncountable number of values in every sub-interval $Q \subset J$; for if $f(Q)$ were countable, then $Q = \bigcup_{y \in f(Q)} (Q \cap f^{-1}(y))$ would be a countable union of $\lambda$-null sets, contradicting the relation $\lambda(Q) > 0$.

Let $Q_{1}, Q_{2}, \dots$ enumerate the sub-intervals of $J$ with rational endpoints. Recursively pick  distinct points $a_{n}, b_{n} \in f(Q_{n})$ avoiding the finite set $\{ a_{i}, b_{i} : i < n \}$; this is possible since $f(Q_{n})$ is uncountable. The set $M := \{ a_{n} : n \geq 1 \}$ is countable (hence, Borel) subset of $F$ with $a_{n} \in M$ and $b_{n} \notin M$ for every $n$. For each $n,$ there exist $x, x' \in Q_{n}$ with $f(x) = a_{n} \in M$ and $f(x') = b_{n} \in F \setminus M$, so $f^{-1} M$ and $f^{-1}(F \setminus M)$ both meet $Q_{n}$. As  rational sub-intervals form the base of $J$, both sets are dense in $J$; hence, every point of $J$ is a boundary point of $f^{-1} M$, i.e., $\partial(f^{-1} M) \supset J$. Therefore, $\lambda(\partial(f^{-1} M)) \geq \lambda(J) > 0$ and $h = \mathds{1}_{f^{-1} M} \notin R[a,b]$.

In every case $G \in R[c,d]$, while $h = G \circ f \notin R[a,b]$. This proves the statement by the contrary, and the necessity follows.
\end{proof}

The assumption, that $f$ is constant on no interval, can be dropped, at the price of ignoring the intervals where $f$ is constant.

\begin{theorem}[Criterion for a fixed internal function, general case]\label{t26}
\hfill

Fix $f \in R[a,b]$ and let $U := \cup \{ (\alpha, \beta) \subset [a,b]: f|_{(\alpha,\beta)} = const \}$.
Then for each $G \in R[c,d], G \circ f \in R[a,b]$ if and only if $f|_{[a,b] \setminus U}$ has $N^{-1}$-property, i.e., $\lambda(f^{-1} N \setminus U) = 0$ for each $N$ with $\lambda(N) = 0$.
\end{theorem}

Note that Theorem \ref{t25} is a special case with $U = \varnothing$.

\begin{proof}
The set $U$ is open, and each $x \in U$ lies in an open interval on which $f$ is constant.

\textit{Sufficiency.} Let $G \in R[c,d]$. For $x \in U$, the function $h = G \circ f$ is constant in a neighborhood of $x$, so $x \notin D(h)$. Together with the inclusion from the proof of Theorem \ref{t22}, this gives
\begin{equation}\label{eq:Dh}
D(h) \subset D(f) \cup (f^{-1} D(G) \setminus U).
\end{equation}
Here $\lambda(D(f)) = 0$ because $f \in R[a,b]$, and $\lambda(f^{-1} D(G) \setminus U) = 0$ because $\lambda(D(G)) = 0$ and $f|_{[a,b] \setminus U}$ has $N^{-1}$-property. Hence $\lambda(D(h)) = 0$, and $h \in R[a,b]$ by Theorem \ref{criteria}.

\textit{Necessity.} We follow the proof of Theorem \ref{t25} and point out the changes. Suppose that there exists a set $N \subset [c,d]$ with $\lambda(N) = 0$ and $\lambda^{*}(f^{-1} N \setminus U) > 0$. Exactly as in the proof of Theorem \ref{t25}, we may and do assume that $N$ is Borel. The function $f$ is Lebesgue measurable, hence
$$\mu(E) := \lambda(f^{-1} E \setminus U), \quad E \in B([c,d]),$$
is a finite Borel measure on $[c,d]$ with $\mu(N) > 0$ and $\lambda(N) = 0$. As in the proof of Theorem \ref{t25}, there is a compact set $F \subset [c,d]$ with $\lambda(F) = 0$ and $\mu(F) > 0$. Put $A = f^{-1} F$ and $B = A \setminus U$, then $\lambda(B) = \mu(F) > 0$. As before, for each Borel $M \subset F$ we have $G := \mathds{1}_{M} \in R[c,d]$ and $D(G \circ f) = \partial (f^{-1} M)$, so it suffices to find a Borel $M \subset F$ with $\lambda(\partial (f^{-1} M)) > 0$.

\textit{Case 1: $\lambda(B \setminus A^{\circ}) > 0$.} Take $M = F$. Then $\partial A \supset A \setminus A^{\circ} \supset B \setminus A^{\circ}$, so $\lambda(\partial A) > 0$.

\textit{Case 2: $\lambda(B \setminus A^{\circ}) = 0$.} Then $\lambda(A^{\circ} \setminus U) \geq \lambda(B \cap A^{\circ}) > 0$, so some component $J$ of the open set $A^{\circ}$ satisfies $\lambda(J \setminus U) > 0$. We have $f(J) \subset F$. Consider a finite measure $\mu_{J}(E) := \lambda(J \cap f^{-1} E \setminus U)$ on $F$, of the total mass $\lambda(J \setminus U) > 0$.

\textit{Case 2(a): $\mu_{J}$ has an atom $y_{0}$.} Put $P = f^{-1} \{ y_{0} \}$, then $\lambda(P \setminus U) \geq \mu_{J}(\{ y_{0} \}) > 0$. The interior $P^{\circ}$ is a union of open intervals on which $f = y_{0}$, so $P^{\circ} \subset U$. Hence $\partial P \supset P \setminus P^{\circ} \supset P \setminus U$ and $\lambda(\partial P) > 0$. Take $M = \{ y_{0} \}$.

\textit{Case 2(b): $\mu_{J}$ is non-atomic.} Let $T$ be the set of points $x \in J$ such that $\lambda(Q \setminus U) > 0$ for each open interval $Q$ with $x \in Q \subset J$. The set $(J \setminus U) \setminus T$ is covered by the countably many sub-intervals $Q \subset J$ with rational endpoints and $\lambda(Q \setminus U) = 0$, so it is a null set, and $\lambda(T) \geq \lambda(J \setminus U) > 0$.

Let $Q_{1}, Q_{2}, \dots$ enumerate the sub-intervals $Q \subset J$ with rational endpoints and $\lambda(Q \setminus U) > 0$. Each $f(Q_{n})$ is uncountable: otherwise $Q_{n} \setminus U = \bigcup_{y \in f(Q_{n})} (Q_{n} \cap f^{-1}(y) \setminus U)$ would be a countable union of $\lambda$-null sets, while $\lambda(Q_{n} \setminus U) > 0$. Choose  points $a_{n}, b_{n} \in f(Q_{n})$ and the set $M := \{ a_{n} : n \geq 1 \} \subset F$ exactly as in the proof of Theorem \ref{t25}. Then $f^{-1} M$ and $f^{-1}(F \setminus M)$ both meet every $Q_{n}$. Let $x \in T$. Each neighborhood of $x$ contains a sub-interval $Q \subset J$ with rational endpoints and $x \in Q$, and $Q = Q_{n}$ for some $n$ by the definition of $T$. Hence $x \in \partial (f^{-1} M)$, i.e., $\partial (f^{-1} M) \supset T$ and $\lambda(\partial (f^{-1} M)) \geq \lambda(T) > 0$.

In every case $G \in R[c,d]$, while $G \circ f \notin R[a,b]$, and the necessity follows.
\end{proof}

\section{Conditions with continuous functions for Riemann integrability}

The next example is close to the one in Gelbaum \& Olmsted (2003) \cite{Gelbaum}, Chap. 8, Section 34.
\begin{example}[If $f$ is monotonic and continuous, then $h$ need not be Riemann integrable]\hfill
\end{example}
Let $G(x)=\mathds{1}_{C}(x), x\in [0,1]$, where $C$ is Cantor's set \cite{Gelbaum}, let
$\varphi:[0,1] \to [0,1], \varphi(x) = \frac{1}{2}(x + \theta(x))$, where $\theta$ is the Cantor staircase function \cite{Gelbaum}, and let $f := \varphi^{-1}$; a similar counterexample is given in \cite{Klippert}.

$\varphi$ is a continuous and strictly increasing bijection of $[0,1]$, hence so is $f$.
$D(G) = C,~ \lambda(D(G)) = 0$, hence   $G \in R[0,1]$.

$h=G\circ f = \mathds{1}_{f^{-1}(C)} = \mathds{1}_{\varphi(C)}$. Here, $\varphi(C)=K$ is the "fat" Cantor   set of the Lebesgue measure $1/2$; see \cite{Gelbaum}, Section 16.

$D(h) = K$, because $K$ is nowhere dense. Then $\lambda(D(h)) = \lambda(K)=\frac{1}{2} \neq 0$, and therefore,  $h \not\in R[0,1]$.

The inverse is essential here: $\varphi$ itself satisfies $|\varphi(x)-\varphi(y)| \geq \frac{1}{2}|x-y|$, so by Lemma \ref{l31} below it has $N^{-1}$-property, and $G\circ \varphi \in R[0,1]$ for each $G \in R[0,1]$ by Theorem \ref{t24}.

\begin{lemma}\label{l31}\hfill

Let $A \subset \mathbb{R}, $
$f:A \to \mathbb{R}$ satisfy the inverse Lipschitz condition, i.e., $$\exists L > 0~ \forall x,y \in A: |f(x) - f(y)| \geq L|x-y|. $$
Then $\forall B \subset A:~ \lambda^{*}(B) \leq \frac{1}{L}\lambda^{*}(f(B)).$
\end{lemma}
\begin{proof}
Let $B \subset A$,
 $f(B) \subset \cup_{n\geq 1} (a_n, b_n]$.
Then $B \subset f^{-1} f(B) \subset \cup_{n\geq 1} f^{-1}(a_n, b_n]. $
$$\forall n \geq 1~ \forall x,y \in f^{-1}(a_n, b_n]: |f(x)-f(y)| \geq L|x-y|,$$
$$|x-y|\leq L^{-1}|f(x)-f(y)|\leq L^{-1}(b_n-a_n).$$

If $f^{-1}(a_n, b_n] = \varnothing$ then  $\lambda(f^{-1}(a_n, b_n])=0 \leq L^{-1}(b_n-a_n)$.

If $f^{-1}(a_n, b_n] \neq \varnothing,$ then
$$\exists c \in f^{-1}(a_n, b_n]~
\forall x \in f^{-1}(a_n, b_n]: |x| \leq |c| + |x-c|\leq |c| + L^{-1}(b_n-a_n). $$
Introduce the bounds $m = \inf f^{-1}(a_n, b_n] $ and $M = \sup f^{-1}(a_n, b_n] $.
We have $f^{-1}(a_n, b_n] \subset [m,M]$, and
$$\lambda^{*}(f^{-1}(a_n, b_n]) \leq \lambda^{*}([m,M]) = M-m = \sup_{x,y\in f^{-1}(a_n, b_n]}|x-y| \leq L^{-1}(b_n-a_n), $$
$$\lambda^{*}(B) \leq \sum_{n \geq 1} \lambda^{*}(f^{-1}(a_n, b_n]) \leq L^{-1} \sum_{n \geq 1} (b_n-a_n),$$
$$\lambda^{*}(B) \leq L^{-1} \inf\{\sum_{n \geq 1} (b_n-a_n) : f(B) \subset \cup_{n\geq 1} (a_n, b_n]\} = L^{-1} \lambda^{*}(f(B)).$$
\end{proof}

\begin{theorem}[Composition of absolutely continuous function and Riemann integrable function] \label{t31}\hfill

If $f \in AC[a,b], f' \neq 0 \pmod \lambda$, and $G \in R[c,d]$, then $h \in R[a,b].$
\end{theorem}
\begin{proof}
$G \in R[c,d],$ hence $\lambda(D(G)) = 0.$ Next, $D(h) \subset D(f) \cup f^{-1} D(G),$ so
$$\lambda^{*}(D(h)) \leq \lambda^{*}(D(f)) + \lambda^{*}(f^{-1}D(G)).$$

Since $f \in AC[a,b]$, we have $f \in C[a,b]$ and $D(f) = \varnothing,$ therefore
$$\lambda^{*}(D(h)) \leq \lambda^{*}(f^{-1}D(G)).$$

From $f' \neq 0 \pmod \lambda$ it follows that there exists $N \subset [a,b]$ with $\lambda(N) = 0$ such that $f'(x)$ exists and $f'(x) \neq 0$ for all $x \in [a,b] \setminus N.$ Define
$$B_{n} := \Big\{ x \in [a,b] : |f(y)-f(x)| \geq \tfrac{|y-x|}{n} \text{ for all } y \in [a,b] \text{ with } |y-x| \leq \tfrac{1}{n} \Big\}.$$

\textit{1. Proof that $[a,b] \setminus N$ is covered by $B_n, n\ge 1.$}

Let $x \in [a,b] \setminus N,$ then there exists a derivative $f'(x)$ with $|f'(x)| = c > 0.$ Since
$$|f'(x)| = \lim_{\substack{y \to x\\ y \in [a,b]}} \Big|\frac{f(y)-f(x)}{y-x}\Big|,$$
there exists $\delta > 0$ such that
$$\forall y \in [a,b],\ 0 < |y-x| < \delta: \Big|\frac{f(y)-f(x)}{y-x}\Big| \geq c - \frac{c}{2}=\frac{c}{2}.$$
Define $n = [\max(\frac{2}{c}, \frac{1}{\delta})] + 1.$ For this $n$ we have $c > \frac{2}{n}$ and $\delta \geq \frac{1}{n}$, hence $x \in B_n.$

\medskip
\textit{2. Proof that $D(h)$ is a null set.}

Consider a partition $[a,b] = \cup_{k=1}^{m_n}I_{nk},$ where $ I_{nk}$  is a closed interval with  length not exceeding $1/n.$

$B_{nk} := B_n \cap I_{nk}, 1 \leq k \leq m_n, $

$f^{-1}D(G) = (f^{-1}D(G) \cap N) \cup (f^{-1}D(G) \cap ([a,b] \setminus N)),$

$f^{-1}D(G) \cap ([a,b] \setminus N) \subset f^{-1}D(G) \cap \cup_{n \geq 1}\cup_{k = 1}^{m_n}B_{nk},$

$0 \leq \lambda^{*}(f^{-1}D(G) \cap N) \leq \lambda^{*}(N) = \lambda(N) = 0,$

$\lambda^{*}(f^{-1}D(G)) \leq \sum_{n \geq 1}\sum_{k=1}^{m_n} \lambda^{*}(f^{-1}D(G) \cap B_{nk}). $
Fix $n \geq 1$ and $k$ with $1 \leq k \leq m_n.$ For any $x,y \in B_{nk}$ both points lie in $I_{nk},$ whose length does not exceed $1/n,$ so $|y-x| \leq 1/n.$ Since $x \in B_n$ and $y \in [a,b]$ with $|y-x| \leq 1/n,$ the defining property of $B_n$ gives $|f(y)-f(x)| \geq |y-x|/n.$ Hence the restriction of $f$ to $B_{nk}$ satisfies the inverse Lipschitz condition with $L = 1/n,$ and Lemma \ref{l31} implies
$$ \lambda^{*}(f^{-1}D(G) \cap B_{nk}) \leq n\lambda^{*}(f(f^{-1}D(G) \cap B_{nk})) \leq n\lambda^{*}(D(G)) = 0, $$
$$ \lambda^{*}(f^{-1}D(G)) \leq \sum_{n \geq 1} \sum_{k=1}^{m_n} 0 = 0,$$ then
$\lambda(D(h)) = 0$ and $h \in R[a,b].$
\end{proof}

 The next statements follow directly from  Theorem \ref{t31}, since $x^2, \sqrt x \in AC[0,1].$
 \begin{example}
 Let $G\in R[0,1].$ Then $G(x^2), G(\sqrt x)\in R[0,1].$
\end{example}
\section{Lebesgue integrability}

\begin{theorem}\label{t41}
Let $f$ be a Lebesgue measurable function on $\mathbb{R}$. Assume the following: $$\exists L>0 ~ \forall x,y \in \mathbb{R} : |f(x)-f(y)| \geq L|x-y|. $$
Then $f$ is $(S,S)$-measurable.
\end{theorem}
\begin{proof}
We have to show that $\forall A \in S: f^{-1} A \in S.$

Let $A \in S$. Then there exist a Borel set $B$ and a Lebesgue measurable set $C$ with $A = B \cup C$ and $\lambda(C) = 0$; see \cite{Halmos}. We get a decomposition
$f^{-1}A = f^{-1}B \cup f^{-1}C.$
Since $f$ is $S$-measurable, $f^{-1}B \in S.$
Lemma \ref{l31} implies that $$\lambda^{*}(f^{-1}C) \leq L^{-1}\lambda^{*}(f(f^{-1}C)) \leq L^{-1}\lambda^{*}(C) = 0.$$
Thus,  $\lambda^{*}(f^{-1}C) = 0,   f^{-1}C \in S$, and
$f^{-1}A \in S.$
\end{proof}



\begin{corollary}[Composition of inverse-Lipschitz function and Lebesgue integrable function] \hfill
\label{c41}

 Let $f$ satisfy the conditions of Theorem \ref{t41} and $G\in L(\mathbb{R}).$    Then $G\circ f \in L(\mathbb{R}).$
\end{corollary}
\begin{proof}
By Theorem \ref{t41} $f$ is $(S,S)$-measurable, $G$ is $(S,B(\mathbb{R}))$-measurable, hence $G\circ f$ is $(S,B(\mathbb{R}))$-measurable, i.e., $h$ is Lebesgue measurable. By the change of variables formula \cite{Halmos} $$\int_{\mathbb{R}} |G(f(x))|~d\lambda(x)=\int_{\mathbb{R}}|G(t)|~d(\lambda f^{-1})(t).$$
Here, $\lambda f^{-1}$ denotes the induced measure, which according to the Lemma \ref{l31} satisfies $(\lambda f^{-1})(A)\leq L^{-1}\lambda (A),~ A\in S.$ Therefore,
$$\int_{\mathbb{R}} |G(f(x))|~d\lambda(x)\le L^{-1} \int_{\mathbb{R}}|G(t)|~d\lambda(t)<\infty.$$ This implies the statement.
\end{proof}
\begin{example}
Let $G\in L(\mathbb{R}).$ Then $G(x+\arctan x)\in L(\mathbb{R}).$
\end{example}

This is a consequence of Corollary \ref{c41}, because by the Lagrange theorem
$$|x+\arctan x -y-\arctan y|\ge |x-y|,~x,y\in \mathbb{R}.$$

\begin{remark}
Let $A\in B(\mathbb{R}), f$ be a Lebesgue measurable function such that
$$\exists L>0 ~\forall x,y\in A: |f(x)-f(y)|\ge L|x-y|,$$
and $G\in L(f(A)).$ Then $G\circ f \in L(A).$
\end{remark}
\begin{proof}
The inverse mapping $f^{-1}$ is continuous. Then $f(A)\in B(\mathbb{R})$ as a preimage of a Borel set $A$ under a continuous mapping $f^{-1}$. Similarly to Theorem \ref{t41} the function $f$ is $(S,S)$-measurable. Then $G\circ f$ is Lebesgue measurable, and
$$\int_A |G(f(x))|~d\lambda(x)=\int_{f(A)} |G(t)|~d(\lambda f^{-1})(t)\le L^{-1}\int_{f(A)} |G(t)|~d\lambda(t)<\infty.$$
\end{proof}
\begin{example}

(a) Let $G(x)/\sqrt x\in L[0,1]$. Then $G(x^2)\in L[0,1].$

(b) Let $xG(x)\in L[0,1].$ Then $G(\sqrt x)\in L[0,1].$
\end{example}
\begin{proof}
In both cases, the condition implies that $G$ is Lebesgue measurable in [0,1].

(a) By Theorem \ref{t41} the function $f(x)=x^2$ is
$(S,S)$-measurable in each interval $[\epsilon, 1], 0<\epsilon<1$, and therefore it is $(S,S)$-measurable in $[0,1].$ Then by  changing the variables  and taking into  account  $\frac{d(\lambda f^{-1})}{d\lambda}=\frac{1}{2}t^{-1/2}, t\in (0,1],$ we obtain

$$\int_{[0,1]}|G(x^2)|~d\lambda(x)=\frac{1}{2}\int_{[0,1]}\frac{|G(t)|}{\sqrt t}~d\lambda(t)<\infty.$$

\medskip
(b) By Theorem \ref{t41} the function $f(x)=\sqrt x$ is $(S,S)$-measurable in $[0,1].$ Then by the change of variables and taking into account that $\frac{d(\lambda f^{-1})}{d\lambda}=2t, t\in [0,1],$ we obtain
$$\int_{[0,1]}|G(\sqrt x)|~d\lambda(x)=2\int_{[0,1]}t|G(t)|~d\lambda(t)<\infty.$$
\end{proof}

It is tempting to extend the statement of Theorem 4.1 to inverse-H\"older functions, i.e., functions $g$ such that
$$\exists L>0,~ \alpha>0 ~ \forall x,y \in \mathbb{R} : |g(x)-g(y)| \geq L|x-y|^\alpha. $$
If $\alpha<1$ this condition implies the inverse-Lipschitz property in each finite interval, and therefore, $g$ will possess the $N^{-1}$-property, and if additionally $g$ is Lebesgue measurable, then it will also be $(S,S)$- measurable. For that reason, we will focus on the case $\alpha>1$.

The next statement is presented in Dovgoshey et al. (2006) \cite{Dovgoshey},  Proposition 10.1.
\begin{lemma}[H\"older continuity of the Cantor staircase $\theta$]\label{lem:theta}\hfill

For all $x,y\in[0,1]$, it holds $$|\theta(x)-\theta(y)|\le |x-y|^{\beta_0}$$ with
$\beta_0:=\log_3 2\in(0,1)$.
\end{lemma}
\begin{example}[Continuous, increasing,  and inverse-H\"older function, which is not (S,S)-measurable and has no $N^{-1}$-property]\hfill

The function $$f(x)=\frac{1}{2}(x+\theta(x)), x\in [0,1],$$
from Example 3.1 is a continuous strictly increasing bijection on $[0,1]$, moreover, for the two Cantor's sets we have $f(C)=K$, and $|f(x)-f(y)|\le |x-y|^{\beta_0}$ as a convex combination of the two functions satisfying this H\"older condition.

\pagebreak

Extend $f$ to a map $g:\mathbb{R}\to\mathbb{R}$ that stays globally H\"older with exponent
$\beta_0$ and remains a bijection of $\mathbb{R}$ (so it must grow \emph{sub-linearly}
outside $[0,1]$):
\[
g(x):=
\begin{cases}
1+(x-1)^{\beta_0}, & x>1,\\[2pt]
\tfrac12\bigl(x+\theta(x)\bigr), & 0\le x\le 1,\\[2pt]
-|x|^{\beta_0}, & x<0.
\end{cases}
\]
It is continuous and strictly increasing, therefore a bijection $\mathbb{R}\to\mathbb{R}$; indeed, the pieces match at the junctions, $g(0^-)=0=f(0)$ and $g(1^+)=1=f(1)$, and $g|_{[0,1]}=f$, so $g(C)=f(C)=K$.

We now check that $g$ is globally H\"older of exponent $\beta_0$. Recall the elementary inequality $|s^{\beta_0}-t^{\beta_0}|\le|s-t|^{\beta_0}$ for $s,t\ge0$ (subadditivity of $t\mapsto t^{\beta_0}$, $\beta_0\in(0,1)$). On each piece $g$ is H\"older of exponent $\beta_0$ with constant $1$: on $[0,1]$ this is the bound $|f(x)-f(y)|\le|x-y|^{\beta_0}$ above; on $[1,\infty)$ one has $|g(x)-g(y)|=|(x-1)^{\beta_0}-(y-1)^{\beta_0}|\le|x-y|^{\beta_0}$; and on $(-\infty,0]$ one has $|g(x)-g(y)|=\bigl||x|^{\beta_0}-|y|^{\beta_0}\bigr|\le|x-y|^{\beta_0}$. For arbitrary $x<y$ we split $[x,y]$ at whichever of the breakpoints $0,1$ it contains; in the extreme case $x<0<1<y$,
\[
\begin{aligned}
|g(y)-g(x)| &\le|g(y)-g(1)|+|g(1)-g(0)|+|g(0)-g(x)|\\
&\le(y-1)^{\beta_0}+1+(-x)^{\beta_0}\le 3|y-x|^{\beta_0},
\end{aligned}
\]
because each segment has length at most $y-x$ and $y-x\ge1$ gives $1\le|y-x|^{\beta_0}$; the cases with fewer breakpoints are identical with one or two terms. Hence
$$|g(y)-g(x)|\le 3|y-x|^{\beta_0}, \quad x, y \in \mathbb{R}.$$
Now, set
\[
\psi=g^{-1}:\mathbb{R}\to\mathbb{R}
\]
which is again a continuous, strictly increasing bijection. Inverting the global bound $|g(u)-g(v)|\le 3|u-v|^{\beta_0}$ (set $x=g(u), y=g(v)$ and raise to the power $1/\beta_0$) yields $$|\psi(x)-\psi(y)|\ge 3^{-\alpha}|x-y|^\alpha, \quad x,y\in \mathbb{R},$$ with $\alpha=\beta_0^{-1}=\log_2 3>1$.

This function does not have $N^{-1}$-property, because $\lambda(C)=0$ but $\lambda(\psi^{-1}C)=\lambda(K)=1/2>0$.

Explain why $\psi$ is not $(S,S)$-measurable. Indeed, the set $K$ of positive Lebesgue measure contains a non-measurable subset $V$, then $\psi(V)\subset \psi(K)=C$, C is a null 
set, hence $\psi(V)\in S.$ But $\psi$ is injective, so $\psi^{-1}(\psi(V))=V\notin S.$
\end{example}

Note that a similar example can be constructed for any inverse-H\"older exponent $\alpha>1$, using instead of $K$ other Cantor's sets with positive Lebesgue measure; see Devos et al. (2025) \cite{Devos} for the H\"older continuity of a generalized Cantor staircase. Thus, for a continuous and inverse-H\"older function $f$ and a Lebesgue measurable $F$, the composition $F\circ f$ need not be Lebesgue measurable. Of course, there is no problem with a Lebesgue measurable internal function and a Borel external one. 

The absolute continuity of an internal function helps to provide the integrability of the composition.
\begin{theorem}[Composition of monotone, absolutely continuous function and Lebesgue measurable function]\hfill

Let a monotone function $f\in AC[a,b]$ with $f'\neq 0$ almost everywhere, $c$ and $d$ be the minimum and maximum values of $f$, respectively, and $$\frac{G(t)}{|f'(f^{-1}(t))|}\in L[c,d].$$ Then $G\circ f \in L[a,b]$ and $$\int_{[a,b]} G(f(x)) d\lambda (x)=\int_{[c,d]} \frac {G(t)}{|f'(f^{-1}(t))|}d\lambda (t).$$

\end{theorem}

\begin{proof}
The function $f$ satisfies the conditions of Theorem \ref{t31}, and from the proof of the theorem it follows that $f$ is $(S,S)$-measurable. This makes it possible to change the variables in the integral of $|G\circ f|.$

Since $f$ is monotone, $f\in AC[a,b]$, and $f'\neq0$ almost everywhere, $f$ cannot be constant on any subinterval (otherwise $f'$ would vanish on a set of positive measure), so $f$ is strictly monotone. Being also continuous, $f$ is a strictly monotone bijection of $[a,b]$ onto $f([a,b])=[c,d]$, and its inverse $f^{-1}$ is a continuous, strictly monotone bijection of $[c,d]$ onto $[a,b].$

Moreover, $f'\neq0$ a.e. means $|f'|>0$ a.e.; by the area formula $\lambda(f(A))=\int_A|f'|\,d\lambda$, valid for a monotone $f\in AC[a,b]$, the map $f$ sends every set of positive measure to a set of positive measure, so $f^{-1}$ maps null sets to null sets. Hence, by the Banach--Zareckii theorem \cite{Bogachev}, p. 388, $f^{-1}\in AC[c,d]$ with $(f^{-1})'(t)=1/f'(f^{-1}(t))$ almost everywhere. Consequently the induced measure $\lambda f^{-1}$ is absolutely continuous w.r.t. $\lambda$ on $[c,d]$ with the Radon-Nikodym derivative $1/|f'(f^{-1}(t))|.$ The latter function is well-defined, because for the null set  $N$ of points where either $f'$ does not exist or it equals 0, the set $\{t: f^{-1}(t)\in N\}$  has Lebesgue measure 0 as a preimage of $N$ under the mapping $f^{-1}.$ This mapping is $(S,S)$-measurable as an absolutely continuous function with the derivative which differs from 0 a.e.

Thus, $$\int_{[a,b]} |G(f(x))| d\lambda (x)=\int_{[c,d]} |G(t)|d(\lambda f^{-1})(t)= \int_{[c,d]} \frac {|G(t)|}{|f'(f^{-1}(t))|}d\lambda (t)<\infty.$$
This implies the statement.
\end{proof}

The latter theorem is stated on a compact interval; the next example on the whole real line follows from it by a limiting argument.

\begin{example}
Let $\alpha>0$ and $G$ be a real function such that $|t|^{(1/\alpha) -1}G(t) \in L(\mathbb{R})$. Then $G(|x|^\alpha sign~ x) \in L(\mathbb{R})$ and $$\int_\mathbb{R} G(|x|^\alpha sign ~x)~d\lambda(x)=\alpha^{-1}\int_\mathbb{R}|t|^{(1/\alpha)-1}G(t) ~d\lambda(t).$$
\end{example}
\begin{proof}
Put $f(x) = |x|^\alpha sign~x$. For $n \geq 1$, the function $f$ is increasing and absolutely continuous on $[-n,n]$, maps it onto $[-n^\alpha, n^\alpha]$, and for $x \neq 0,$ $f'(x) = \alpha |x|^{\alpha-1} \neq 0$ . Since $f^{-1}(t) = |t|^{1/\alpha} sign~t$, we get $1/|f'(f^{-1}(t))| = \alpha^{-1} |t|^{(1/\alpha)-1}$ for $t \neq 0$. The function $\alpha^{-1}|t|^{(1/\alpha)-1}|G(t)|$ belongs to $L[-n^\alpha, n^\alpha]$, so the latter theorem applied to $|G|$ and to $G$ on $[-n,n]$ gives
\begin{align*}
\int_{[-n,n]} |G(f(x))|~d\lambda(x) &= \alpha^{-1}\int_{[-n^\alpha,n^\alpha]} |t|^{(1/\alpha)-1}|G(t)|~d\lambda(t) \\
&\leq \alpha^{-1}\int_\mathbb{R} |t|^{(1/\alpha)-1}|G(t)|~d\lambda(t) < \infty
\end{align*}
and the same identity holds for $G$ in place of $|G|$. Letting $n \to \infty$, the monotone convergence theorem yields $G \circ f \in L(\mathbb{R})$, and then the dominated convergence theorem, applied to both sides of the identity for $G$, yields the stated formula.
\end{proof}

\section{Conclusion}

We found conditions which ensure that a composition of functions is Riemann- or Lebesgue integrable. In particular, Theorems 2.3, 2.5, and 2.6 provide nontrivial criteria for the Riemann integrability. Various examples and counterexamples are presented. The authors intend to expand the  results to  multiple integrals.

\begin{acknowledgement}
The authors thank Dr. G. Riabov (Kyiv) and MS student O. Liubimov (Kyiv) for fruitful discussions.
\end{acknowledgement}

\begin{acknowledgement}[title={Conflict of interest}]
The authors declare no conflict of interest.
\end{acknowledgement}

\begin{funding}
The authors declare that during the preparation of this manuscript no funds, grants or other support were withdrawn.
\end{funding}

\begin{acknowledgement}[title={Author contribution}]
All authors have contributed equally to the paper.
\end{acknowledgement}
\bibliographystyle{VMSTA/bib/vmsta2-mathphys}
\bibliography{main}

\end{document}